\documentclass[reqno,12pt]{amsart}

\usepackage{amsmath, amsthm, amssymb, enumerate, amscd}
\usepackage[mathscr]{eucal}
\usepackage{pgf,tikz}
\usepackage{tikz, tikz-cd}
\usetikzlibrary{matrix,decorations.pathmorphing,arrows}

\theoremstyle{plain}
\newtheorem{theorem}{Theorem}[section]
\newtheorem{proposition}[theorem]{Proposition}
\newtheorem{corollary}[theorem]{Corollary}

\theoremstyle{definition}

\newtheorem{remark}[theorem]{Remark}

\newtheorem{examples}[theorem]{Examples}

\newcommand{\N}{{\mathbb{N}}}

\begin{document}

\title[Decomposition of an element of a numerical semigroup]{The number of addends in the decomposition of an element of a numerical semigroup into atoms}

\thanks{}
\author{Hamid Kulosman}
\address{Department of Mathematics\\ 
University of Louisville\\
Louisville, KY 40292, USA}
\email{hamid.kulosman@louisville.edu}



\date{}

\begin{abstract} 
We prove that for every nonempty set $\Sigma$ of integers bigger than $1$, which has at most three elements, there exists a numerical semigroup $T$ and an element $x$ of $T$ such that a natural number $n$ is the number of atoms in a decomposition of $x$ into atoms if and only if $n$ belongs to $\Sigma$. We also propose three related conjectures.
\end{abstract}

\maketitle

\section{Introduction}\label{intro}
Let's start with a description of a factorization process of a nonzero nonunit element $x$ of an integral domain $R$. If $x$ is an atom, we stop. If not, we decompose it as $x=x_0x_1$, where both $x_0$ and $x_1$ are nonzero nonunits. If both $x_0, x_1$ are atoms, we stop. If not, we take the first from the left of the elements $x_0, x_1$ which is reducible and decompose it as a product of two nonzero nonunits. Say $x_0$ is reducible. We decompose it: $x_0=x_{0,0}x_{0,1}$. Now we have $x= x_{0,0}x_{0,1}x_1$. If all of the $x_{0,0}$, $x_{0,1}$, $x_1$ are atoms, we stop. If not, we take the first from the left of the elements  $x_{0,0}$, $x_{0,1}$, $x_1$ which is reducible and decompose it as a product of two nonzero nonunits. Say $x_{0,0}$ is an atom and $x_{0,1}$ is reducible: $x_{0,1}=x_{0,1,0}x_{0,1,1}$. Now we have $x=x_{0,0}x_{0,1,0}x_{0,1,1}x_1$, etc. We call this process a {\it factorization process} of $x$. If it stops after finitely many steps, we say that this is a {\it finite factorization process} of $x$. If it never stops, we say that this is an {\it infinite factorization process} of $x$. Obviously, this process depends only on the multiplication in $R$, so it can be considered in a multiplicative monoid and, hence, an analogue can be considered in an additive monoid $(T,+)$. There we call it an {\it addendization process} as we decompose $x$ into a sum of addends. We have {\it finite} and {\it infinite addendization processes}. Addendization processes of additive monoids shed some light on the structure of monoids, but also give some insights about the structure of integral domains since to each additive monoid $T$ we can associate a monoid domain $F[X;M]$, where $F$ is a field. So it is of interest to investigate them in detail.

Consider now a special case when $(T,+)$ is a numerical semigroup. Then there are only finitely many atoms in $T$ and also every addendization process is finite. For a given nonzero nonunit element $x\in T$ we are interested in a number of atoms that we obtain after the last step of an addendization process. We call the set of all such numbers that we get for a given element $x\in T$ the {\it addendization set} of $x$ and we denote it by $AS(x)$.

\begin{examples}\label{ex_addendization_numbers}
(1) If $T=\N=\{0,1,2,3,\dots\}$, the only atom is $1$ and we have $AS(x)=\{x\}$ for every $x\in T\setminus \{0\}$.

(2) If $T=\langle 2,3\rangle=\{0,2,3,4,5,6,\dots\}$, the only atoms are $2$ and $3$, and we have:
\begin{align*}
AS(2) &= \{1\},\\
AS(3) &= \{1\},\\
AS(4) &= \{2\},\\
AS(5) &= \{2\},\\
AS(6) &= \{2,3\},\\
AS(7) &= \{3\},\\
AS(8) &= \{3,4\},\\
AS(9) &= \{3,4\},\\
AS(10) &= \{4,5\}.
\end{align*}

(3) If $T=\langle 3,7,8\rangle=\{3,6,7,8,9,10,\dots\}$, the only atoms are $3,7$ and $8$. The smallest $x\in T$ such that $AS(x)$ has one element is of course $x=3$: $AS(3)=\{1\}$. The smallest $x\in T$ such that $AS(x)$ has two elements is $x=14$: $AS(14)=\{2,3\}$. The smallest $x\in T$ such that $AS(x)$ has three elements is $x=21$: $AS(21)=\{3,4,7\}$.
\end{examples}

Note that for an element $x$ of a numerical semigroup $T$, $AS(x)=\{1\}$ if and only if $x$ is an atom. Also, if a subset $\Sigma$ of $\N^\ast=\{1,2,3,\dots\}$ contains $1$ and at least one more element, then there is no $x\in T$ such that $AS(x)=\Sigma$. The \underbar{question} we are interested in is the following one: {\it given a subset $\Sigma$ of the set $\N\setminus \{0,1\}$, does there necessarily exist a numerical semigroup $T$ and an element $x\in T$ such that $AS(x)=\Sigma$}?

\medskip
After giving some notation and preliminaries in Section \ref{notation_prelims}, we will work on this question in Section \ref{addendization_sets}. 

\section{Notation and preliminaries}\label{notation_prelims}
We begin by recalling some definitions and statements. All the notions in this paper that we use  but not define, as well as the statements that we mention or use but don't prove, can be found in one of the books \cite{gil} by R.~Gilmer, \cite{g} by P.~A.~Grillet, and \cite{rgs} by J.~C.~Rosales and P.~A.~Garc\'ia-S\'anchez. The reader can also consult our paper \cite{gk} (joint with R.~Gipson), where the notion of a factorization process in an integral domain was introduced.

\medskip
We use $\subseteq$ to denote inclusion, $\subset$ to denote strict inclusion and $|\cdot|$ to denote the cardinality of a set. We also denote $\N=\{0,1,2,\dots\}$ and $\N^\ast=\{1,2,\dots\}$.

\medskip
An {\it integral domain} is a commutative ring $R\ne \{0\}$ (with multiplicative identity $1$) such that for any $x,y\in R$, $xy=0$ implies $x=0$ or $y=0$.  All the rings that we use in this paper are assumed to be commutative and with multiplicative identity $1$. In fact, all of them will be integral domains.
An element $a\in R$ is said to be a {\it unit} of $R$ if it has a multiplicative inverse in $R$. A non-zero non-unit element $a\in R$ is said to be {\it irreducible} (and called an {\it atom} of $R$) if  $a=bc$ $(b,c\in R$) implies that at least one of the elements $b,c$ is a unit. An integral domain $R$  is said to be {\it atomic} if every non-zero non-unit element of $R$ can be written as a finite product of atoms.

\medskip
Many notions related to factorization in an integral domain can be already defined in the underlying multiplicative monoid $(R,\cdot)$. Hence we can generalize them by defining them in an arbitrary commutative monoid with zero $(M,\cdot)$, and then also in an arbitrary additive monoid $(M,+)$ by just translating the terminology from the multiplicative one to the additive one. The next definitions illustrate this point of view.

\medskip
A {\it commutative monoid}, written additively, is a non\-emp\-ty set $M$ with an operation $+:M\times M\to M$ which is associative, commutative, and has an identity element called {\it zero} ( i.e., an element $0\in M$ such that $a+0=a$ for every $a\in M$. All the monoids used in this paper are assumed to be commutative and written additively. From now on we call them just {\it monoids}. An element $a\in M$ is said to be a {\it unit} of $M$ if it has an additive inverse in $M$. A non-unit element $a\in M$ is said to be {\it irreducible} (and called an {\it atom} of $M$) if  $a=b+c$ $(b,c\in M)$ implies that at least one of the elements $b,c$ is a unit. A monoid $M$  is said to be {\it atomic} if every non-infinity non-unit element of $M$ can be written as a finite sum of atoms.

\medskip
A {\it numerical semigroup} is a submonoid $T$ of $\N$ with finite complement in $\N$. It is finitely generated and has a unique {\it minimal system of generators} $A\subseteq \N$. The cardinality of the minimal system of generators is bounded from above by the least positive element of $T$. If $A$ is a nonempty subset of $\N$, then the submonoid $\langle A\rangle$ of $\N$, generated by $A$, is a numerical semigroup if and only if $\gcd(A)=1$. 

Note that the minimal generating set of a numerical semigroup $T$ consists precisely of all the atoms of $T$. Hence, when we say that $T=\langle n_1, n_2,\dots, n_k\rangle$ is a numerical semigroup, to prove that $A=\{ n_1, n_2,\dots, n_k\}$ is the minimal generating set of $T$ is equivalent to proving each of the elements $n_1, n_2, \dots, n_k$ is an atom of $T$.

\medskip
If $M$ is a  monoid and $F$ is a field, we will consider the {\it monoid ring} $F[X;M]$ associated with $M$. It consists of the polynomial expressions
\[f=a_0X^{\alpha_0}+a_1X^{\alpha_1}+\dots +a_nX^{\alpha_n},\]
where $n\ge 0$, \,$a_i\in F$, \,and $\alpha_i\in M$ $(i=0,1,\dots, n)$, with ``standard'' addition and multiplication. The factorization properties of $F[X;M]$ and the addendization properties of $M$ can be expressed via each other.

\section{Addendization sets}\label{addendization_sets}
In this section we will be discussing the question raised at the end of Section 1. We start with the following simple statement.

\begin{proposition}\label{Sigma_singleton}
For any singleton $\Sigma=\{n\}\subseteq \N\setminus\{0,1\}$ there is a numerical semigroup $T$ and an element $x\in T$ such that $AS(x)=\Sigma$. 
\end{proposition}
\begin{proof}
Take $T=\langle 2,2n-1\rangle$ and $x=2n$.
\end{proof}

Now we consider the case of a two-element subset $\Sigma$ of $\N\setminus \{0,1\}$.

\begin{proposition}\label{Sigma_two_elements}
Consider the subset $\Sigma=\{2,n\}$ of the set $\N\setminus \{0,1\}$, with $n\ge 3$. The monoid 
\[T=\langle k, k+n, kn-(k+n)\rangle,\]
where $k\ge 7$ and $\gcd(n,k)=1$, is a numerical semigroup with minimal generating set $A=\{k,k+n,kn-(k+n)\}$, whose element $x=kn$ has the property $AS(x)=\Sigma$. 
\end{proposition}
\begin{proof}
We first show that each of the elements $k, kn, kn-(k+n)$ is an atom. Clearly $k<k+n$, while the inequality 
\begin{equation}\label{eq_order_atoms}
k+n<kn-(k+n)
\end{equation}
follows from the inequalities $\displaystyle{{2n \over n-2} \le 6<k}$, which are true since $n\ge 3$ and $k\ge 7$. Now $k$ is obviously an atom as the smallest element of a generating set. Also $k+n$ is an atom since $k+n\ne \lambda k$ (for $\lambda\in\N$) as $n \not\equiv 0\pmod k$. Suppose that $kn-(k+n)$ is not an atom. Then 
\begin{equation}\label{eq_non_atom}
kn-(k+n) =\lambda k+\mu (k+n) \quad (\lambda, \mu\in\N).
\end{equation}
This equality (when considered modulo $k$) implies $\mu+1\equiv 0 \pmod k$. Hence
\begin{equation}\label{eq_mu}
\mu\in \{k-1, 2k-1, 3k-1,\dots\}.
\end{equation}
The equality (\ref{eq_non_atom}) can be written as 
\begin{equation}\label{eq_mu_1}
kn=\lambda k+\mu (k+n) +k+n.
\end{equation}
One of the addends on the right hand side (RHS) is $\mu n$, hence  which because of (\ref{eq_mu}) implies $\mu=k-1$. If we plug in this into (\ref{eq_mu_1}) we get $0=\lambda+k$, a contradiction. Thus $kn-(k+n)$ is also an atom.

Now we analyze in which ways $x=kn$ can be decomposed into a sum of atoms. One way is $x=\underline{k+n}+\underline{kn-(k+n)}$, where the underlined expressions are atoms. To find all other ways we write
\begin{equation}\label{kn_decompositions}
\alpha k+\mu(k+n)+\gamma[kn-(k+n)]=k+n+[kn-(k+n)].
\end{equation} 
If here $\beta\ge 2$, then the left hand side (LHS) is bigger than the right hand side (RHS) due to (\ref{eq_order_atoms}). Hence $\gamma=0$ or $1$. If $\gamma=1$, we have $\alpha k+\beta (k+n)=k+n$. Hence $\alpha=0$ and $\beta=1$  (however we already mentioned this way to represent $kn$ as a sum of atoms), or $\beta=0$, in which case we get $k+n=\lambda k$, but that is not possible since $k+n$ is an atom. Finally if $\gamma=0$, we get from (\ref{kn_decompositions}) the equality 
\begin{equation*}
\alpha k+\beta (k+n)=k+n+[kn-(k+n)].
\end{equation*}
Here $\beta=0$ (otherwise $\alpha k+(\beta-1)(k+n)=kn-(k+n)$, which is not possible since $kn-(k+n)$ is an atom). Hence $\alpha k=k+n+[kn-(k+n)]$ and so $\alpha=n$. Thus we got one more way to decompose $x=kn$ as a sum of atoms. 

In total we got the following two ways to decompose $x$ as a sum of atoms:
\begin{align*}
x &= \underline{k+n}+\underline{kn-(k+n)},\\
   &= n\cdot \underline{k},
\end{align*}
where the underlined elements are atoms, so that indeed $AS(x)=\{2,n\}$.
\end{proof}

Next we consider the case of a three-element subset $\Sigma$ of $\N\setminus \{0,1\}$.

\begin{proposition}\label{Sigma_three_elements}
Consider the subset $\Sigma=\{2,n,t\}$ of the set $\N\setminus \{0,1\}$, with $n\ge 3$ and $t\ge n+1$. The monoid 
\[T=\langle tn^2,\, tn^2+n,\, t^2n+1,\, t^2n+n+1,\, t^2n^2-t^2n-1\rangle\]
is a numerical semigroup with minimal generating set $A=\{tn^2, \,tn^2+n,\, t^2n+1,\, t^2n+n+1, \,t^2n^2-t^2n-1\}$, whose element $x=t^2n^2+n$ has the property $AS(x)=\Sigma$. 
\end{proposition}
\begin{proof}
Let $k=t^2n+1$. Then all the hypothesis of Proposition \ref{Sigma_two_elements} are satisfied, namely:
\begin{align*}
n &\ge 3,\\
\gcd(k,n) = \gcd(t^2n&+1, n) = \gcd(k,n)= 1,\\
k=t^2n+1 \ge &\,4^2\cdot 3+1 = 49\ge 7.
\end{align*}
Consider the submonoid 
\[T'=\langle t^2n+1,\, t^2n+n+1,\, t^2n^2-t^2n-1\rangle\]
of $\N$. By Proposition \ref{Sigma_two_elements}, $T'$ is a numerical semigroup with minimal generating set $A'=\{ t^2n+1,\, t^2n+n+1,\, t^2n^2-t^2n-1\}$, whose element $x=kn=t^2n^2+n$ has the property $AS(x)=\Sigma'=\{2,n\}$. We want to extend the generating set $A'$ of $T'$ by adding to it two elements $a,b$ so that the following conditions hold:

\smallskip
(1) $a<b<t^2n+1<t^2n+n+1<t^2n^2-t^2n-1$;

\smallskip
(2) there are $\lambda,\mu\in\N$ such that $\lambda+\mu=t$ and $\lambda a +\mu b=kn=t^2n^2+n$;

\smallskip
(3) the set $A=\{a, \,b,\, t^2n+1,\, t^2n+n+1, \,t^2n^2-t^2n-1\}$ is a minimal generating set of the submonoid $T=\langle tn^2,\, tn^2+n,\, t^2n+1,\, t^2n+n+1,\, t^2n^2-t^2n-1\rangle$ of $\N$, which is a numerical semigroup.

\smallskip
Once we find such $a$ and $b$, we hope that then, additionally, the element $x=t^2n^2+n$ satisfies $AS(x)=\Sigma$. 

\smallskip
In order to get an idea how to define $a$ and $b$ we note that
\begin{align*}
t^2n^2+n=\lambda a+\mu b &> (\lambda+\mu)a=ta,\\
t^2n^2+n=\lambda a+\mu b &< (\lambda+\mu)b=tb,
\end{align*}
from where we get
\begin{align*}
a &< (tn+{1 \over t})\,n,\\
b &> (tn+{1 \over t})\,n.
\end{align*}
This motivates us to try 
\begin{align*}
a &=tn^2,\\
b &=(tn+1)\,n.
\end{align*}
It turns out that this choice of $a$ and $b$ will do.

\smallskip
Now we have $T=\langle tn^2,\, tn^2+n,\, t^2n+1,\, t^2n+n+1,\, t^2n^2-t^2n-1\rangle$. Note that 
\[tn^2+n<t^2n+1 \Leftrightarrow n<tn(t-n)+1,\]
which is true since $t-n\ge 1$ and $t\ge 7$. Thus the above condition (1) holds.

\smallskip
\underline{Claim 1:} {\it The elements $tn^2,\, tn^2+n,\, t^2n+1,\, t^2n+n+1,\, t^2n^2-t^2n-1$ are (the only) atoms of the monoid $T$ and $T$ is a numerical semigroup.}

\smallskip
\underline{Proof of Claim 1:} Clearly $tn^2$ and $tn^2+n$ are atoms. We now prove that $t^2n+1$ is an atom. Otherwise, 
\[t^2n+1=\alpha\, tn^2+\beta\,(tn^2+n)\]
for some $\alpha, \beta\in\N$, which is not possible modulo $n$. Next we prove that $t^2n+n+1$ is an atom. Suppose to the contrary. Then
\begin{equation}\label{eq_third_atom}
t^2n+n+1=\alpha\, tn^2+\beta\,(tn^2+n)+\gamma\,(t^2n+1).
\end{equation}
Here $\gamma=0$ is not possible modulo $n$. Hence $\gamma\ge 1$. 

\smallskip
\underline{1$^\circ$ case: Either $\alpha$ or $\beta$ is $\ne 0$.} Then $\alpha\,tn^2+\beta\,(tn^2+n)>n$, hence the RHS of the equality (\ref{eq_third_atom}) is bigger than the LHS, a contradiction.

\smallskip
\underline{2$^\circ$ case: $\alpha=\beta=0$.} Then (\ref{eq_third_atom}) gives $t^2n+n+1=\gamma\,(t^2n+1)$. If $\gamma\ge 2$, then the RHS of this equality is bigger than the LHS, a contradiction. Hence $\gamma=1$. Then we get $t^2n+n+1=t^2+1$, a contradiction.

\smallskip
So (\ref{eq_third_atom}) is not possible and thus $t^2n+n+1$ is an atom.

Next we show that $t^2n^2-t^2n-1$ is an atom. Suppose to the contrary. Then:
\begin{equation}\label{eq_fourth_atom}
t^2n^2=t^2n-1=\alpha\,tn^2+\beta\,(tn^2+n)+\gamma\,(t^2n+1)+\delta\,(t^2n+n+1).
\end{equation}
Here $\gamma=\delta=0$ is not possible modulo $n$. Hence either $\gamma\ge 1$ or $\delta\ge 1$. Also $\alpha=\beta=0$ is not possible because of the properties of the numerical semigroup $T'$ (see Proposition \ref{Sigma_two_elements}). Hence either $\alpha\ge 1$ or $\beta\ge 1$. Modulo $n$ we get from (\ref{eq_fourth_atom}) that $-1\equiv \gamma+\delta \pmod n$, i.e., $\gamma+\delta\equiv n-1 \pmod n$. Hence $\gamma+\delta\in\{n-1, 2n-1, 3n-1,\dots\}$. Hence we get from (\ref{eq_fourth_atom}):
\begin{align*}
\mathrm{RHS} &\ge (n-1)(t^2n+1)+tn^2\\
        &= t^2n^2-t^2n+n-1+tn^2\\
        &>t^2n^2-t^2n-1\\
        &=\mathrm{LHS},
\end{align*}
a contradiction. Thus $t^2n^2-t^2n-1$ is an atom.

\smallskip
\underline{Claim 1 is proved.}

\medskip
\underline{Claim 2:} {\it Let $x=kn=t^2n^2+n$. Then AS(x)=\{2,n,t\}.}

\smallskip
\underline{Proof of Claim 2:} We have
\begin{align*}
x &= \underline{k}\,n\\
   &= \underline{(t^2n^2-t^2n-1)} + \underline{(t^2n+n+1)}\\
   &= (t-1)\,\underline{tn^2}+1\cdot\underline{(tn^2+n)},
\end{align*}
where the underlined expressions are atoms, so that 
\[AN(x)\supseteq \{2,n,t\}.\]
Consider the equality
\begin{align}\label{eq_AS}
(t^2n+1)\,n=\alpha\,tn^2&+\beta\,(tn^2+n) + \gamma\,(t^2n+1)\notag\\
                                                                         &+\delta\,(t^2n+n+1)+\varepsilon\,(t^2n^2-t^2n-1)
\end{align}
with $\alpha, \beta,\gamma,\delta,\varepsilon\in\N$. If $\alpha=\beta=0$, then it follows from the properties of the numerical semigroup $T'$ (see Proposition \ref{Sigma_two_elements}) that either $\gamma=0$, $\delta=\varepsilon=1$, or $\delta=\varepsilon=0$, $\gamma=n$, which are two of the three decompositions listed above. Suppose now that either $\alpha\ne 0$ or $\beta\ne 0$. 

\smallskip
\underline{1$^\circ$ case: $\varepsilon\ne 0$.} Then $\varepsilon=1$ (otherwise, if $\varepsilon>1$, the RHS of (\ref{eq_AS}) would be bigger than the LHS). We have:
\begin{align*}
t^2n^2+n = \alpha\,tn^2&+\beta\,(tn^2+n) + \gamma\,(t^2n+1)\\
                                                                         &+\delta\,(t^2n+n+1)+ t^2n^2-t^2n-1,
\end{align*}
hence
\[t^2n+n+1=\alpha\,tn^2+\beta\,(tn^2+n) + \gamma\,(t^2n+1)+\delta\,(t^2n+n+1).\]
Hence $\delta=0$ (otherwise, as $\alpha\ne 0$ or $\beta\ne 0$, we would have that the RHS of this equality is bigger than the LHS). Now we have
\[t^2n+n+1=\alpha\,tn^2+\beta\,(tn^2+n) + \gamma\,(t^2n+1).\]
Hence $\gamma=1$ (otherwise, as $\alpha\ne 0$ or $\beta\ne 0$, we would have that the RHS of this equality is bigger than the LHS). So we have
\[t^2n+n+1=\alpha\,tn^2+\beta\,(tn^2+n) + t^2n+1,\]
hence $n=\alpha\,tn^2+\beta\,(tn^2+n)$, a contradiction.

\smallskip
\underline{2$^\circ$ case: $\varepsilon= 0$.} We have:
\[(t^2n+1)\,n = \alpha\,tn^2+\beta\,(tn^2+n) + \gamma\,(t^2n+1)+\delta\,(t^2n+n+1).\]

\smallskip
\underline{2A$^\circ$ case: $\gamma\ne 0$ or $\delta\ne 0$.} We have $0\equiv \gamma+\delta \pmod n$, hence $\gamma+\delta\in\{n, 2n, 3n, \dots\}$, hence (as $\alpha\ne 0$ or $\beta\ne 0$) the RHS of this equality is bigger than the LHS, a contradiction. 

\smallskip
\underline{2B$^\circ$ case: $\gamma=\delta=0$.} We have $(t^2n+1)\,n = \alpha\,tn^2+\beta\,(tn^2+n)$, hence $t^2n+1 = \alpha\,tn+\beta\,(tn+1)$. If $\alpha+\beta\ge t+1$, then we get from this equality
\[\mathrm{RHS} > (t+1)\,tn=t^2n+tn>\mathrm{LHS},\]
a contradiction. If $\alpha+\beta\le t-1$, then
\begin{align*}
\mathrm{RHS} &= \alpha\,tn^2+\beta\,(tn^2+n)\\
                        &<(\alpha+\beta)\,(tn^2+n)\\
                        &\le (t-1)\,(tn^2+n)\\
                        &=t^2n^2-tn^2+tn-n\\
                        &\le t^2n^2+n
\end{align*}
(as $t^2n^2=tn^2+tn-n\le t^2n^2+n \Leftrightarrow tn\le 2n+tn^2 \Leftrightarrow t\le tn+2$). Thus the RHS is smaller than the LHS, a contradiction. Finally, if $\alpha+\beta=t$, then $(t^2n+1)\,n = \alpha\,tn^2+(t-\alpha)\,(tn^2+n)$, hence $t^2n^2+n = \alpha\,tn^2+t^2n^2+tn-\alpha\,tn^2-\alpha$, hence $(1+\alpha)\,n=tn$. This implies $\alpha=t-1$ and 
$\beta=1$. This is the third one of the three decompositions listed at the beginning of the proof of this claim.

\smallskip
\underline{Claim 2 is proved.}

\smallskip
The claims 1 and 2 contain all the facts from the statement of the proposition.
\end{proof}

\begin{remark}\label{choice_of_k}
At the beginning of the proof of the previous proposition we define $k=t^2n+1$. Our first attempt was $k=tn+1$. The ``$+1$'' part of this formula for $k$ is put so that $k\equiv 1 \pmod n$, which later, after also putting $a\equiv 0 \pmod n$ and $b\equiv 0 \pmod n$, makes the proofs of atomicity and addendization in certain cases very simple. However, with $k=tn+1$ we wern't able to rule out the case $x=\lambda\,x+\mu\,b$ with $\lambda+\mu<t$. We then tried $k=ztn+1$, where $z$ is a paremeter, and this worked for various $z$, in particular for $z=t$.
\end{remark}

\begin{corollary}\label{Sigma_any_two_elt_set}
Consider the subset $\Sigma''=\{n,t\}$ of the set $\N\setminus \{0,1\}$, with $n\ge 3$ and $t\ge n+1$. The monoid 
\[T''=\langle tn^2,\, tn^2+n,\, t^2n+1\rangle\]
is a numerical semigroup with minimal generating set $A''=\{tn^2, \,tn^2+n,\, t^2n+1\}$, whose element $x=t^2n^2+n$ has the property $AS(x)=\Sigma''$. 
\end{corollary}
\begin{proof}
The monoid $T''$ is a submonoid of the monoid $T$ from the previous proposition. For any $d\in\N^\ast$, if $d\mid tn^2$ and $d\mid tn^2+n$, then $d\mid n$, hence $d\nmid t^2n+1$. So the gcd of the three generators is $1$. It follows from the previous proposition that $T''$ is a numerical semigroup with the minimal generating set $A''=\{tn^2, \,tn^2+n,\, t^2n+1\}$. By the previous proposition the element $x=t^2n^2+n$ can be decomposed into atoms in three different ways in $T$, however only two of them work in $T''$ as well. Hence $AS(x)=\Sigma''$.
\end{proof}

\begin{corollary}\label{Sigma_any_three_elt_set}
Consider the subset $\Sigma'''=\{r+1,\,n,\,t\}$ of the set $\N\setminus \{0,1\}$, with $r\ge 1$, $n\ge r+2$ and $t\ge n+1$. The monoid 
\[T'''=\langle rtn^2,\, r(tn^2+n),\, r(t^2n+1),\, r(t^2n+n+1),\, t^2n^2-t^2n-1\rangle\]
is a numerical semigroup with minimal generating set $A'''=\{rtn^2, \,r(tn^2+n),\, r(t^2n+1),\, r(t^2n+n+1), \,t^2n^2-t^2n-1\}$, whose element $x'''=r(t^2n^2+n)$ has the property $AS(x''')=\Sigma'''$. 
\end{corollary}
\begin{proof}
Denote $a=tn^2,\,b=tn^2+n,\,c=t^2n+1,\, d=t^2n+n+1,\, e=t^2n^2-t^2n-1$ and $A=rtn^2,\, B=r(tn^2+n),\,C=r(t^2n+1),\, D=r(t^2n+n+1)$. We also use the notation $x=t^2n^2+n$ from the previous proposition. Taking into account the previous proposition, in order to prove that $A<B<C<D<e$ we only need to prove that $D<e$. We have:
\begin{align*}
e=t^2n^2-t^2n-1 &\ge t^2n(r+2)-t^2n-1\\
                               &= t^2n(r+1)-1\\
                               &= rt^2n+t^2n-1\\
                               &= rt^2n+(r+3)^2n-1\\
                               &= rt^2n+r^2n+6rn+9n-1\\
                               &> rt^2n+rn+r\\
                               &= D.
\end{align*}
Taking into account the previous proposition, in order to prove that $A, B, C, D, e$ are atoms of $T'''$, we only need to prove that $e$ is an atom of $T'''$. Suppose to the contrary. Then
\begin{align*}
e&= \alpha\,A+\beta\,B+\gamma\,C+\delta\,D\\
  &= (\alpha\,r)a+(\beta\,r)b+(\gamma\,r)c+(\delta\,r)d,
\end{align*}
contradicting the fact that $e$ is an atom of the numerical semigroup $T=\langle a,b,c,d,e\rangle$ from the previous proposition.  

\medskip
There are three decompositions of $x$ into atoms in $T$. Here are the three analogous decompositions of $x'''$ in $T'''$:
\begin{align*}
x''' &= n\,\underline{C},\\
     &= (t-1)\,\underline{A}+\underline{B},\\
     &= \underline{D}+r\underline{e},
\end{align*}
where the underlined symbols are atoms of $T'''$. Hence
\[AS(x''')\supseteq\{r+1,\, t,\,n\}.\]
We now show that there is no other way to decompose $x'''$ into atoms in $T'''$. Consider the relation
\begin{equation}\label{eq_decompositions}
x'''=\alpha\,A+\beta\,B+\gamma\,C+\delta\,D+\varepsilon\,e.
\end{equation}

\smallskip
\underline{1$^\circ$ case: $\varepsilon=0$.} Then (\ref{eq_decompositions}) becomes $x'''=\alpha\,A+\beta\,B+\gamma\,C+\delta\,D$, which is (after cancelling $r$) equivalent to  $x=\alpha\,a+\beta\,b+\gamma\,c+\delta\,d$ in $T$. By the previous proposition we have exactly two options: 
\begin{align*}
x&=n\,\underline{c},\\
  &=(t-1)\,\underline{a}+\underline{b},
\end{align*}
where the underlined symbols are atoms of $T$. Thus we got the first two of the above listed three options for a decoomposition of $x'''$ into atoms in $T'''$.

\smallskip
\underline{2$^\circ$ case: $\varepsilon=\varepsilon'\,r$, where $\varepsilon'\ge 1$.} Then after cancelling $r$, the relation (\ref{eq_decompositions}) is equivalent to  $x=\alpha\,a+\beta\,b+\gamma\,c+\delta\,d+\varepsilon'\,e$ in $T$, with 
$\varepsilon'\ge 1$. By the previous proposition we have $\alpha=\beta=\gamma=0$, $\delta=1$, $\varepsilon'=1$. Hence $\varepsilon=r$. Thus we got the above listed third decomposition of $x'''$ into atoms in $T'''$.

\smallskip
\underline{3$^\circ$ case: $\varepsilon=\varepsilon'\,r+r'$, where $r'\in\{1,2,\dots, r-1\}$.} Then $rx=r\alpha\,a+r\beta\,b+r\gamma\,c+r\delta\,d+r\varepsilon'\,e+r'\,(t^2n^2-t^2n-1)$ in $T$. This is not possible modulo $r$ as all the terms of this equality, but the last one, are divisible by $r$.
\end{proof}

Finally we formulate a theorem which unifies all of the above results.

\begin{theorem}\label{main_theorem}
Given a subset $\Sigma$ of the set $\N\setminus \{0,1\}$, such that $|\Sigma|\le 3$, there exists a numerical semigroup $T$ and an element $x\in T$ such that $AS(x)=\Sigma$.
\end{theorem}
\begin{proof}
If $|\Sigma|=1$ the statement follows from Proposition \ref{Sigma_singleton}. If $|\Sigma|=2$ the statement follows from Proposition \ref{Sigma_two_elements} and Corollary \ref{Sigma_any_two_elt_set}.  If $|\Sigma|=3$ the statement follows from Corollary \ref{Sigma_any_three_elt_set}. 
\end{proof}

\begin{examples}\label{examples}
(1) By Proposition \ref{Sigma_two_elements}, the addendization set $\Sigma=\{2,3\}$ can be realized in the numerical semigroup $T=\langle 7,10,11\rangle$ as $AS(21)$. We have:
\begin{align*}
21&=\underline{10}+\underline{11}\\
    &=3\cdot\underline{7}.
\end{align*} 

(2) By Corollary \ref{Sigma_any_two_elt_set}, the addendization set $\Sigma=\{3,4\}$ can be realized in the numerical semigroup $T=\langle 36,39,49\rangle$ as $AS(147)$. We have:
\begin{align*}
147&=3\cdot \underline{49}\\
    &=3\cdot\underline{36}+\underline{39}.
\end{align*} 

(2) By Corollary \ref{Sigma_any_two_elt_set}, the addendization set $\Sigma=\{3,4\}$ can be realized in the numerical semigroup $T=\langle 36,39,49\rangle$ as $AS(147)$. We have:
\begin{align*}
147&=3\cdot \underline{49}\\
      &=3\cdot\underline{36}+\underline{39}.
\end{align*} 

(3) By Corollary \ref{Sigma_three_elements}, the addendization set $\Sigma=\{2,3,4\}$ can be realized in the numerical semigroup $T=\langle 36,39,49,52,95\rangle$ as $AS(147)$. We have:
\begin{align*}
147&=\underline{52}+\underline{95}\\
      &=3\cdot \underline{49}\\
      &=3\cdot\underline{36}+\underline{39}.
\end{align*} 

(4) By Corollary \ref{Sigma_any_three_elt_set}, the addendization set $\Sigma=\{3,5,7\}$ can be realized in the numerical semigroup $T=\langle 350, 360,492,502,979\rangle$ as $AS(2460)$. We have:
\begin{align*}
2460&=5\cdot \underline{492}\\     
        &=6\cdot\underline{350}+\underline{360}\\
        &=\underline{502}+2\cdot \underline{979}.
\end{align*} 

Our examples are probably not minimal in any sense, however they provide formulas for constructing $T$ and $x$ that work for any given $\Sigma$ with $|\Sigma|\le 3$.
\end{examples}

\section{Concluding remarks}\label{concluding_remarks}
 We intuitively feel that Theorem \ref{main_theorem} holds as well if we omit the condition $|\Sigma|\le 3$, i.e., that the following conjecture holds:

\medskip
{\bf Conjecture 1.} {\it Given any nonempty finite subset $\Sigma$ of the set $\N\setminus \{0,1\}$, there exists a numerical semigroup $T$ and an element $x\in T$ such that $AS(x)=\Sigma$.}

\medskip
We cannot imagine any obstacle that would prevent this stement to be true, however, we are not able to construct such a $T$ and $x\in T$. 

\medskip
We mentioned that our examples are probably not minimal in any sense, so it would be interesting to ask, when $\Sigma$ is given, how to construct minimal $T$ and $x$ in various senses. For example, for $\Sigma=\{2,3\}$ a minimal $T$ and $x$, in a certain precisely defined sense, would be $T=\langle 2,3\rangle$ and $x=6$. Also for $\Sigma=\langle 2,3,4\rangle$ a minimal $T$ and $x$, in a certain precisely defined sense, would probably be $T=\langle 9,12,13,23\rangle$ and $x=36$.

\medskip
Furthermore, we intuitively feel that the following two equivalent conjectures (which are beyond the scope of numerical semigroups) hold:

\medskip
{\bf Conjecture 2.} {\it Given any nonempty subset $\Sigma$, finite or infinite, of the set $(\N\cup\{\infty\})\setminus \{0,1\}$, there exists an additive monoid $M$ and an element $x\in M$ such that $AS(x)=\Sigma$.}

\medskip
{\bf Conjecture 3.} {\it There exists an additive monoid $M$ such that for any nonempty subset $\Sigma$, finite or infinite, of the set $(\N\cup\{\infty\})\setminus \{0,1\}$, there exists an element $x=x(\Sigma)\in M$ such that $AS(x)=\Sigma$.}

\medskip
Note that the statement $\infty\in AS(x)$ for some $x\in M$ means that $x$ has an infinite addendization process.

Clearly Conjecture 3 implies Conjecture 2. Conversely, if Conjecture 2 holds, then using a direct sum of monoids one can show that Conjecture 3 holds.


\begin{thebibliography}{99}                   
\bibitem{gil}
GILMER, R.: 
    \textit{Commutative Semigroup Rings},
    The University of Chicago Press, Chicago, 1984.
\bibitem{gk}    
GIPSON, R., KULOSMAN, H.:
    \textit{Atomic and AP semigroup rings $F[X;M]$, where $M$ is a submonoid of the additive monoid of nonnegative rational numbers},
    Intern. Electr. J. Algebra, \textbf{22}(2017), 133-146.
\bibitem{g}      
GRILLET, P.A.:
    \textit{Commutative Semigroups},
    Advances in Mathematics vol. 2, Springer Science+Business Media Dordrecht 2001.             
\bibitem{rgs}      
ROSALES, J.C., GARC\'IA-S\'ANCHEZ, P.A.:
    \textit{Numerical Semigroups},
    Developments in Mathematics vol. 20, Springer Science+Business Media, LLC 2009.      
\end{thebibliography}
\end{document}